\documentclass[12pt]{amsart}
\usepackage{a4wide,amsmath,amssymb,graphicx}
\usepackage{color,enumerate,txfonts}
\usepackage{bbm,mathrsfs,stmaryrd,extarrows}
\allowdisplaybreaks

\let\pa\partial
\let\na\nabla
\let\eps\varepsilon
\newcommand{\R}{\mathbb{R}}
\newcommand{\N}{\mathbb{N}}

\newtheorem{theorem}{Theorem}
\newtheorem{proposition}[theorem]{Proposition}
\newtheorem{lemma}[theorem]{Lemma}

\theoremstyle{remark}

\begin{document}
\title[A note on Aubin-Lions-Dubinski\u{\i} lemmas]{A note on
Aubin-Lions-Dubinski\u{\i} lemmas}

\author[X. Chen]{Xiuqing Chen}
\address{School of Sciences, Beijing University of Posts and Telecommunications,
  Beijing, 100876 China; Institute for Analysis and Scientific Computing, Vienna University of
	Technology, Wiedner Hauptstra\ss e 8--10, 1040 Wien, Austria}
\email{buptxchen@yahoo.com}

\author[A. J\"ungel]{Ansgar J\"ungel}
\address{Institute for Analysis and Scientific Computing, Vienna University of
	Technology, Wiedner Hauptstra\ss e 8--10, 1040 Wien, Austria}
\email{juengel@tuwien.ac.at}

\author[J.-G. Liu]{Jian-Guo Liu}
\address{Department of Physics and Department of Mathematics, Duke University
  Durham, NC 27708, USA}
\email{jian-guo.liu@duke.edu}

\date{\today}

\thanks{The first author acknowledges support from the National Science
Foundation of China, grant 11101049.
The second author acknowledges partial support from
the Austrian Science Fund (FWF), grants P20214, P22108, I395, and W1245.
This research was supported by the European Union under
Grant Agreement number 304617 (Marie-Curie Project ``Novel Methods in
Computational Finance'').
The third author acknowledges support from the National Science Foundation of
the USA, grant DMS 10-11738.
}

\begin{abstract}
Strong compactness results for families of functions in seminormed nonnegative cones
in the spirit of the Aubin-Lions-Dubinski\u{\i} lemma
are proven, refining some recent results in the literature.
The first theorem sharpens slightly a result of Dubinski\u{\i} (1965) for
seminormed cones.
The second theorem applies to piecewise constant functions in time and sharpens
slightly the results of Dreher and J\"ungel (2012) and Chen and Liu (2012).
An application is given, which is useful in the study of porous-medium or
fast-diffusion type equations.
\end{abstract}



\maketitle


\section{Introduction}\label{sec.intro}

The Aubin-Lions lemma
states criteria under which a set of functions is relatively compact in
$L^p(0,T;B)$, where $p\ge 1$, $T>0$, and $B$ is a Banach space.
The standard Aubin-Lions lemma states that if $U$ is bounded in $L^p(0,T;X)$ and
$\pa U/\pa t=\{\pa u/\pa t:u\in U\}$ is bounded in $L^r(0,T;Y)$, then
$U$ is relatively compact in $L^p(0,T;B)$,
under the conditions that
\begin{equation*}
  X\hookrightarrow B \mbox{ compactly}, \quad
  B\hookrightarrow Y \mbox{ continuously},
\end{equation*}
and either $1\le p<\infty$, $r=1$ or
$p=\infty$, $r>1$. Typically, when $U$ consists of approximate solutions to an
evolution equation, the boundedness of $U$ in $L^p(0,T;X)$ comes from suitable
a priori estimates, and the boundedness of $\pa U/\pa t$ in $L^r(0,T;Y)$
is a consequence of the evolution equation at hand. The compactness is needed
to extract a sequence in the set of approximate solutions,
which converges strongly in $L^p(0,T;B)$. The limit is expected to be a
solution to the original evolution equation, thus yielding an existence result.

In recent years, nonlinear counterparts of the Aubin-Lions lemma were shown
\cite{BaSu12,ChLi12,Mai03}.
In this note, we aim to collect these results, which are scattered in the literature,
and to prove some refinements.
In particular, we concentrate on the case in which the set $U$ is bounded
in $L^p(0,T;M_+)$, where $M_+$ is a nonnegative cone (see below). This situation
was first investigated by Dubinski\u{\i}, and therefore, we call the
corresponding results Aubin-Lions-Dubinski\u{\i} lemmas.

Before detailing our main results, let us review the compactness theorems in
the literature.
The first result on the compact embedding of spaces of Banach space valued functions was
shown by Aubin in 1963 \cite{Aub63}, extended by Dubinski\u{\i} in 1965 \cite{Dub65},
also see \cite[Th\'eor\`eme 5.1, p.~58]{Lio69}. Some unnecessary assumptions
on the spaces were removed by Simon in his famous paper \cite{Sim87}.
The compactness embedding result was sharpened by Amann \cite{Ama00} involving
spaces of higher regularity, and by Roub\'{\i}\v{c}ek, assuming that
the space $Y$ is only locally convex Hausdorff \cite{Rou90}
or that $\pa U/\pa t$ is bounded in the space of vector-values measures
\cite[Corollary 7.9]{Rou05}. This condition can be replaced by a boundedness
hypothesis in a space of functions with generalized bounded variations
\cite[Prop.~2]{Kal13}. A result on compactness in $L^p(\R;B)$ can be found in
\cite[Theorem 13.2]{Tem95}.

The boundedness of $U$ in $L^p(0,T;B)$
can be weakened to tightness of $U$ with respect to a certain lower semicontinuous
function; see \cite[Theorem 1]{RoSa03}. Also the converse of the Aubin-Lions lemma
was proved (see \cite{RaTe01} for a special situation).

Already Dubinski\u{\i} \cite{Dub65} observed that
the space $X$ can be replaced by a seminormed set, which can be interpreted as a
nonlinear version of the Aubin-Lions lemma. (Recently, Barrett and S\"uli
\cite{BaSu12} corrected an oversight in Theorem 1 of \cite{Dub65}.) Furthermore,
the space $B$ can be replaced by $K(X)$, where $K:X\to B$ is a compact operator,
as shown by Maitre \cite{Mai03}, motivated by the nonlinear compactness result
of Alt and Luckhaus \cite{AlLu83}.

Instead of boundedness of $\pa U/\pa t$ in $L^r(0,T;Y)$, the condition
on the time shifts
$$
  \|\sigma_h u-u\|_{L^p(0,T-h;Y)}\to 0\quad\mbox{as }h\to 0, \quad
	\mbox{uniformly in }u\in U,
$$
where $(\sigma_h u)(t)=u(t+h)$, can be imposed to achieve compactness
\cite[Theorem 5]{Sim87}. If the functions $u_\tau$ in $U$ are piecewise constant
in time with uniform time step $\tau>0$, this assumption was simplified
in \cite{DrJu12} to
$$
  \|\sigma_\tau u_\tau-u_\tau\|_{L^r(0,T-h;Y)}\le C\tau,
$$
where $C>0$ does not depend on $\tau$. This condition avoids
the construction of linear interpolations of $u_\tau$
(also known as Rothe functions \cite{Kac85}). Nonlinear versions
were given in \cite{ChLi12}, generalizing the results of Maitre.

In the literature, discrete versions of the Aubin-Lions lemma were
investigated. For instance, compactness properties for a discontinuous and continuous
Galerkin time-step scheme were shown in \cite[Theorem 3.1]{Wal10}. In \cite{GaLa12},
compactness to sequences of functions obtained by a Faedo-Galerkin approximation
of a parabolic problem was studied.

In this note, we generalize some results of \cite{ChLi12,DrJu12} (and \cite{GaLa12})
to seminormed nonnegative cones.
We call $M_+$ a {\em seminormed nonnegative cone} in a Banach space $B$ if
the following conditions hold:
$M_+\subset B$; for all $u\in M_+$ and $c\ge 0$, $cu\in M_+$; and if there
exists a function $[\cdot]:M_+\to[0,\infty)$ such that $[u]=0$ if and only if
$u=0$, and $[cu]=c[u]$ for all $c\ge 0$. We say that $M_+\hookrightarrow B$
continuously, if there exists $C>0$ such that $\|u\|_B\le C[u]$ for all
$u\in M_+\subset B$.
Furthermore, we write $M_+\hookrightarrow B$ compactly, if for any bounded sequence
in $M_+$, there exists a subsequence converging in $B$.

\begin{theorem}[Aubin-Lions-Dubinski\u{\i}]\label{thm.dub}
Let $B$, $Y$ be Banach spaces and $M_+$ be a seminormed nonnegative cone in $B$
with $M_+\cap Y\neq\emptyset$. Let $1\le p\le\infty$. We assume that
\begin{enumerate}[{\rm (i)}]
\item $M_+\hookrightarrow B$ compactly.
\item For all $(w_n)\subset B$,
$w_n\to w$ in $B$, $w_n\to 0$ in $Y$ as $n\to\infty$ imply that $w=0$.
\item $U\subset L^p(0,T;M_+\cap Y)$ is bounded in $L^p(0,T;M_+)$.
\item $\|\sigma_h u-u\|_{L^p(0,T-h;Y)}\to 0$ as $h\to 0$, uniformly in $u\in U$.
\end{enumerate}
Then $U$ is relatively compact in $L^p(0,T;B)$ (and in $C^0([0,T];B)$ if
$p=\infty)$.
\end{theorem}

This result generalizes slightly Theorem 3 in \cite{ChLi12}.
The novelty is that
we do {\em not} require the continuous embedding $B\hookrightarrow Y$.
If both $B$ and $Y$ are continuously embedded in a metric space
(such as some Sobolev space with negative index) or in the space of distributions
${\mathcal D}'$, which is naturally satisfied in nearly all applications,
then condition (ii) clearly holds. Therefore, we do not need to check the
continuous embedding $B\hookrightarrow Y$, which is sometimes not obvious, like in
\cite[pp.~1206-1207]{ChLi13}, where $B$ is an $L^1$ space with a
complicated weight and $Y$ is related to a Sobolev space with negative index.
Thus, this generalization is not only interesting in functional analysis but
also in applications.

The proof of Theorem \ref{thm.dub} is motivated by Theorem 3.4 in \cite{GaLa12}
and needs a simple but new idea.
Taking the proof of Theorem 5 in \cite{Sim87} as an example, we compare the
traditional proof and our new idea.
For this, we first list some statements:
\begin{align}
  & B\hookrightarrow Y \mbox{ continuously},\label{glem4} \\
  & X\hookrightarrow B \mbox{ compactly},\label{glem3} \\
  & X\hookrightarrow Y \mbox{ compactly},\label{glem6} \\
  &\forall\,\varepsilon>0,\,\exists\,C_\varepsilon>0,\,\forall u\in X,\;
	\|u\|_B\le \varepsilon \|u\|_X +C_\varepsilon\|u\|_{Y},\label{glem6+0} \\
  & U\; \mbox{is a bounded subset of}\; L^p(0,T;X),\label{glem7} \\
  & \|\sigma_hu-u\|_{L^p(0,T-h;Y)}\rightarrow0\;\mbox{as}\;h\rightarrow0,
	\;\mbox{uniformly for}\;u\in U,\label{glem8}\\
  & \|\sigma_hu-u\|_{L^p(0,T-h;B)}\rightarrow0\;\mbox{as}\;h\rightarrow0,
	\;\mbox{uniformly for}\;u\in U,\label{glem9}\\
  & U\; \mbox{is relatively compact in}\; L^p(0,T;Y),\label{glem10} \\
  & U\; \mbox{is relatively compact in}\; L^p(0,T;B).\label{glem11}
\end{align}

Simon proves \eqref{glem11} \cite[Theorem 5]{Sim87} using the steps
\begin{enumerate}[{\rm I.}]
\item Theorem 5 in \cite{Sim87}: \eqref{glem4}, \eqref{glem3}, \eqref{glem7},
  \eqref{glem8} $\Rightarrow$ \eqref{glem11}.
\item Lemma 8 in \cite{Sim87}: \eqref{glem4}, \eqref{glem3}
  $\Rightarrow$ \eqref{glem6+0}.
\item Theorem 3 in \cite{Sim87}: \eqref{glem3}, \eqref{glem7}, \eqref{glem9}
  $\Rightarrow$ \eqref{glem11}, or \eqref{glem6}, \eqref{glem7}, \eqref{glem8}
	$\Rightarrow$ \eqref{glem10}.
\end{enumerate}

More precisely, 
\begin{align*}
& \hspace*{4mm}\mbox{\em Traditional proof of I}:
  \hspace{40mm} \mbox{\em New proof of I}: \\
& \left.\begin{array}{ll}
\hspace{2.2cm}
\eqref{glem4}, \eqref{glem3}\xLongrightarrow{\mbox{II}}&\eqref{glem6+0}\\
&\eqref{glem7}\\
\left.\begin{array}{ll}
      \eqref{glem4}, \eqref{glem3}\Longrightarrow\!\!&\eqref{glem6}\\
      &\eqref{glem7}\\
      &\eqref{glem8}
      \end{array}\right\}\xLongrightarrow{\mbox{III}}
&\eqref{glem10}
\end{array}\right\}\Longrightarrow\eqref{glem11}; \quad\quad
\left.\begin{array}{ll}
&\eqref{glem3}\\
&\eqref{glem7}\\
\left.\begin{array}{ll}
      \eqref{glem4}, \eqref{glem3}\xLongrightarrow{\mbox{II}}&\eqref{glem6+0}\\
      &\eqref{glem7}\\
      &\eqref{glem8}
      \end{array}\right\}\Longrightarrow
&\eqref{glem9}
\end{array}\right\}\xLongrightarrow{\mbox{III}}\eqref{glem11}.
\end{align*}

In the traditional proof of I, the step \eqref{glem4}, \eqref{glem3} $\Rightarrow$
\eqref{glem6} depends on the continuous embedding \eqref{glem4}. Hence, in that proof,
\eqref{glem4} is essential. In our new proof, only step II:
\eqref{glem4}, \eqref{glem3} $\Rightarrow$ \eqref{glem6+0} depends on \eqref{glem4},
which can be replaced by condition (ii) of Theorem \ref{thm.dub}.
This condition follows from \eqref{glem4} and hence, it is weaker than \eqref{glem4}.

If $U$ consists of piecewise constant functions in time $(u_\tau)$
with values in a Banach space, condition (iv) in Theorem \ref{thm.dub}
can be simplified.
The main feature is that it is sufficient to verify one uniform estimate
for the time shifts $u_\tau(\cdot+\tau)-u_\tau$ instead of all time shifts
$u_\tau(\cdot+h)-u_\tau$ for $h>0$.

\begin{theorem}[Aubin-Lions-Dubinski\u{\i} for piecewise constant functions in time]
\label{thm.dub.pw}\
Let $B$, $Y$ be Ba\-nach spaces and $M_+$ be a seminormed nonnegative cone in $B$.
Let either $1\le p<\infty$, $r=1$ or $p=\infty$, $r>1$. Let
$(u_\tau)\subset L^p(0,T;M_+\cap Y)$ be a sequence of functions,
which are constant on each subinterval $((k-1)\tau,k\tau]$, $1\le k\le N$, $T=N\tau$.
We assume that
\begin{enumerate}[{\rm (i)}]
\item $M_+\hookrightarrow B$ compactly.
\item For all $(w_n)\subset B$,
$w_n\to w$ in $B$, $w_n\to 0$ in $Y$ as $n\to\infty$ imply that $w=0$.
\item $(u_\tau)$ is bounded in $L^p(0,T;M_+)$.
\item There exists $C>0$ such that for all $\tau>0$,
$\|\sigma_\tau u_\tau-u_\tau\|_{L^r(0,T-\tau;Y)}\le C\tau$.
\end{enumerate}
Then, if $p<\infty$, $(u_\tau)$ is relatively compact in $L^p(0,T;B)$ and if
$p=\infty$, there exists a subsequence of $(u_\tau)$ converging in $L^q(0,T;B)$ for
all $1\le q<\infty$ to a limit function belonging to $C^0([0,T];B)$.
\end{theorem}

This result generalizes slightly Theorem 1 in \cite{DrJu12} and
Theorem 3 in \cite{ChLi12} (for piecewise constant functions in time). The proof in
\cite{DrJu12} is {also} based on a characterization of the norm of Sobolev-Slobodeckii
spaces. Our proof just uses elementary estimates for the difference
$\sigma_\tau u_\tau-u_\tau$ and thus simplifies the proof in \cite{DrJu12}.
Note that Theorems \ref{thm.dub} and \ref{thm.dub.pw} are also valid
if $M_+$ is replaced by a seminormed cone or Banach space.
We observe that for functions $u_\tau(t,\cdot)=u_k$ for $t\in((k-1)\tau,k\tau]$,
$1\le k\le N$, the estimate of (iv) can be formulated in terms of the
difference $u_{k+1}-u_k$ since
$$
  \|\sigma_\tau u_\tau-u_\tau\|^r_{L^r(0,T-\tau;B)}
	= \sum^{N-1}_{k=1}\int_{(k-1)\tau}^{{k\tau}}\|u_{k+1}-u_k\|^r_{B}dt
  = \tau\sum^{N-1}_{k=1}\|u_{k+1}-u_k\|^r_{B}.
$$

A typical application is the cone of nonnegative functions $u$
with $u^m\in W^{1,q}(\Omega)$, which occurs in diffusion equations
involving a porous-medium or fast-diffusion term.
Applying Theorem \ref{thm.dub.pw}, we obtain the following result.

\begin{theorem}\label{thm.app}
Let $\Omega\subset\R^d$ $(d\ge 1)$ be a bounded domain with $\pa\Omega\in C^{0,1}$.
Let $(u_\tau)$ be a sequence of nonnegative
functions which are constant on each subinterval $((k-1)\tau,k\tau]$,
$1\le k\le N$, $T=N\tau$.
Furthermore, let
$0<m<\infty$, $\gamma\ge 0$, $1\le q<\infty$, and $p\ge\max\{1,\frac{1}{m}\}$.
\begin{enumerate}[{\rm (a)}]
\item If there exists $C>0$ such that for all $\tau>0$,
$$
  \|\sigma_\tau u_\tau-u_\tau\|_{L^1(0,T-\tau;(H^\gamma(\Omega))')}
	+ \|u_\tau^m\|_{L^p(0,T;W^{1,q}(\Omega))} \le C,
$$
then $(u_\tau)$ is relatively compact in $L^{mp}(0,T;L^{mr}(\Omega))$,
where $r\ge \frac{1}{m}$ is such that
$W^{1,q}(\Omega)$ $\hookrightarrow L^r(\Omega)$ is compact.
\item If additionally $\max\{0,(d-q)/(dq)\}<m<1+\min\{0,(d-q)/(dq)\}$ and
\begin{equation}\label{ulogu}
  \|u_\tau\log u_\tau\|_{L^\infty(0,T;L^1(\Omega))} \le C
\end{equation}
for some $C>0$ independent of $\tau>0$, then $(u_\tau)$ is relatively compact
in $L^p(0,T;L^s(\Omega))$ with $s=qd/(qd(1-m)+d-q)>1$.
\end{enumerate}
\end{theorem}

Part (a) of this theorem generalizes Lemma 2.3 in \cite{CHJ12},
in which only relative compactness in $L^{m\ell}(0,T;L^{mr}(\Omega))$
for $\ell < p$ and $q=2$ was shown. Part (b) improves part (a)
for $m<1$ by allowing for relative compactness in $L^p$ with respect to time
instead of the larger space $L^{mp}$. It generalizes Proposition 2.1
in \cite{HiJu11} in which $m=\frac12$ and $p=q=2$ was assumed. Its proof
shows that the bound on
$u_\tau\log u_\tau$ can be replaced by a bound on $\phi(u_\tau)$,
where $\phi$ is continuous and convex.

The additional estimate \eqref{ulogu} is typical for solutions of semidiscrete
nonlinear diffusion equations for which $\int_\Omega u_\tau\log u_\tau dx$
is an entropy (Lyapunov functional)
with $\int_\Omega|\na u_\tau^m|^2 dx$ as the corresponding entropy
production (see, e.g., \cite[Lemma 3.1]{CHJ12}). Theorem \ref{thm.app}
improves standard compactness arguments. Indeed, let $\frac{1}{m}\leq q<d$.
The additional estimate yields boundedness of $(u_\tau)$ in
$L^\infty(0,T;L^1(\Omega))$. Hence, $\na u_\tau = \frac{1}{m}u_\tau^{1-m}\na u_\tau^m$
is bounded in $L^{p}(0,T;L^\alpha(\Omega))$ with $\alpha=q/(1+q(1-m))$.
Thus, $(u_\tau)$ is bounded in $L^{p}(0,T;W^{1,\alpha}(\Omega))
\hookrightarrow L^{p}(0,T;L^s(\Omega))$.
By the Aubin-Lions lemma \cite{DrJu12}, $(u_\tau)$ is relatively compact in
$L^p(0,T;L^\beta(\Omega))$ for all $\beta<s$. Part (b) of the above theorem
improves this compactness to $\beta=s$ under the condition that $(u_\tau\log u_\tau)$
is bounded in $L^\infty(0,T;L^1(\Omega))$.

This note is organized as follows. In Section \ref{sec.pro}, Theorems
\ref{thm.dub}-\ref{thm.app} are proved. Section \ref{sec.coro}
is concerned with additional results.


\section{Proofs}\label{sec.pro}

\subsection{Proof of Theorem \ref{thm.dub}}

The proof of Theorem \ref{thm.dub} is based on the following
Ehrling type lemma.

\begin{lemma}\label{lem.dub}
Let $B$, $Y$ Banach spaces and $M_+$ be a seminormed nonnegative cone in $B$.
We assume that
\begin{enumerate}[{\rm (i)}]
\item $M_+\hookrightarrow B$ compactly.
\item For all $(w_n)\subset B$,
$w_n\to w$ in $B$, $w_n\to 0$ in $Y$ as $n\to\infty$ imply that $w=0$.
\end{enumerate}
Then for any $\eps>0$, there exists $C_\eps>0$ such that for all
$u$, $v\in M_+\cap Y$,
$$
  \|u-v\|_B \le \eps([u]+[v]) + C_\eps\|u-v\|_Y.
$$
\end{lemma}

\begin{proof}
The proof is performed by contradiction.
Suppose that there exists $\eps_0>0$ such that for all $n\in\N$, there exist
$u_n$, $v_n\in M_+\cap Y$ such that
\begin{equation}\label{aux1}
  \|u_n-v_n\|_B > \eps_0([u_n]+[v_n]) + n\|u_n-v_n\|_Y.
\end{equation}
This implies that $[u_n]+[v_n]>0$ for all $n\in\N$ since otherwise,
$[u_m]=[v_m]=0$ for a certain $m\in\N$ would lead to $u_m=v_m=0$ which
contradicts \eqref{aux1}. Define
$$
  \tilde u_n = \frac{u_n}{[u_n]+[v_n]}, \quad \tilde v_n = \frac{v_n}{[u_n]+[v_n]}.
$$
Then $\tilde u_n$, $\tilde v_n\in M_+\cap Y$ and $[\tilde u_n]\le 1$,
$[\tilde v_n]\le 1$. Taking into account the compact embedding $M_+\hookrightarrow B$,
there exist subsequences of $(\tilde u_n$) and $(\tilde v_n)$, which are
not relabeled, such that $\tilde u_n\to u$ and $\tilde v_n\to v$ in B and hence,
\begin{equation}\label{aux2}
  \tilde u_n-\tilde v_n\to u-v \quad\mbox{in }B.
\end{equation}
We infer from \eqref{aux1} that $\|\tilde u_n-\tilde v_n\|_B
> \eps_0+n\|\tilde u_n-\tilde v_n\|_Y$. This shows, on the one hand, that
$\|\tilde u_n-\tilde v_n\|_B>\eps_0$ and, by \eqref{aux2}, $\|u-v\|_{B}>\eps_0$.
On the other hand, using the continuous embedding $M_+\hookrightarrow B$,
$$
  \|\tilde u_n-\tilde v_n\|_Y \le \frac{1}{n}\|\tilde u_n-\tilde v_n\|_B
	\le \frac{C}{n}([\tilde u_n]+[\tilde v_n]) \le \frac{2C}{n}.
$$
Consequently, $\tilde u_n-\tilde v_n\to 0$ in $Y$. Together with \eqref{aux2},
condition (ii) implies that $u-v=0$, contradicting $\|u-v\|>\eps_0$.
\end{proof}

\begin{proof}[Proof of Theorem \ref{thm.dub}]
First, we prove that
\begin{equation}\label{aux3}
  \|\sigma_h u-u\|_{L^p(0,T-h;B)}\to 0\quad\mbox{as }h\to 0, \quad
	\mbox{uniformly in }u\in U.
\end{equation}
Indeed, by condition (iii), there exists $C>0$ such that $\|u\|_{L^p(0,T;M_+)}\le C$
for all $u\in U$.
Lemma \ref{lem.dub} shows that for any $\eps>0$, there exists $C_\eps>0$ such that
for all $0<h<T$, $u\in U$, and $t\in[0,T-h]$,
$$
  \|u(t+h)-u(t)\|_B \le \frac{\eps}{4C}\big([u(t+h)]+[u(t)]\big)
	+ C_\eps\|u(t+h)-u(t)\|_Y.
$$
Integration over $t\in(0,T-h)$ then gives
\begin{align*}
  \|\sigma_h u-u\|_{L^p(0,T-h;B)}
	&\le \frac{\eps}{2C}\|u\|_{L^p(0,T;M_+)} + C_\eps\|\sigma_h u-u\|_{L^p(0,T-h;Y)} \\
	&\le \frac{\eps}{2} + C_\eps\|\sigma_h u-u\|_{L^p(0,T-h;Y)}.
\end{align*}
We deduce from condition (iv) that for $\eps_1=\eps/(2C_{\eps})$,
there exists $\eta>0$ such that for all $0<h<\eta$ and $u\in U$,
$\|\sigma_h u-u\|_{L^p(0,T-h;Y)}\le\eps_1$. This shows that
$\|\sigma_h u-u\|_{L^p(0,T-h;B)}\le\eps/2+\eps/2=\eps$, proving the claim.

Because of condition (iii) and \eqref{aux3}, the assumptions
of Lemma 6 in \cite{ChLi12} are satisfied, and the desired compactness result
follows. In Lemma 6, only the (compact) embedding $M_+\hookrightarrow B$ is needed.
Let us mention that this lemma is a consequence of a nonlinear
Maitre-type compactness result \cite[Theorem 1]{ChLi12}
(see Proposition \ref{prop.maitre}), which itself uses Theorem 1 in \cite{Sim87}.
\end{proof}


\subsection{Proof of Theorem \ref{thm.dub.pw}}

The proof of Theorem \ref{thm.dub.pw} is based on an estimate of
the time shifts $\sigma_h u_\tau-u_\tau$.

\begin{lemma}\label{lem.pw}
Let $1\le p\le\infty$ and let $u_\tau\in L^p(0,T;B)$ be piecewise constant
in time, i.e., $u_\tau(t)=u_k$ for $(k-1)\tau<t\le k\tau$, $k=1,\ldots,N$, $T=N\tau$.
Then, for $0<h<T$,
$$
  \|\sigma_h u_\tau-u_\tau\|_{L^p(0,T-h;B)} \le h^{1/p}\sum_{k=1}^{N-1}
	\|u_{k+1}-u_k\|_B
	= \frac{h^{1/p}}{\tau}\|\sigma_\tau u_\tau-u_\tau\|_{L^1(0,T-\tau;B)}.
$$
\end{lemma}

\begin{proof}
Denoting by $H$ the Heaviside functions, defined by $H(t)=0$ for $t\le 0$ and
$H(t)=1$ for $t>0$, we find that
$$
  u_\tau(t) = u_1 + \sum_{k=1}^{N-1}(u_{k+1}-u_k)H(t-k\tau), \quad 0<t<T.
$$
This gives
$$
  u_\tau(t+h)-u_\tau(t) = \sum_{k=1}^{N-1}(u_{k+1}-u_k)\big(H(t+h-k\tau)-H(t-k\tau)\big),
	\quad 0<t<T-h,
$$
and
\begin{equation}\label{aux4}
  \|\sigma_\tau u_\tau-u_\tau\|_{L^p(0,T-h;B)}
	\le \sum_{k=1}^{N-1}\|u_{k+1}-u_k\|_B\|H(t+h-k\tau)-H(t-k\tau)\|_{L^p(0,T-h)}.
\end{equation}
If $1\le p<\infty$, we have for $1\le k\le N-1$,
\begin{align*}
  \|H(t&+h-k\tau)-H(t-k\tau)\|_{L^p(0,T-h)}^p
	\le \int_{-\infty}^\infty\big|H(t+h-k\tau)-H(t-k\tau)\big|^p dt \\
	&= \left(\int_{-\infty}^{k\tau-h}+\int_{k\tau-h}^{k\tau}+\int_{k\tau}^\infty\right)
	\big|H(t+h-k\tau)-H(t-k\tau)\big|^p dt
	= \int_{k\tau-h}^{k\tau} dt = h.
\end{align*}
If $p=\infty$, we infer that
\begin{align*}
  \|H(t+h-k\tau)-H(t-k\tau)\|_{L^\infty(0,T-h)}
	&\le \|H(t+h-k\tau)-H(t-k\tau)\|_{L^\infty(\R)} \\
	&= \|H(t+h-k\tau)-H(t-k\tau)\|_{L^\infty(k\tau-h,k\tau)} = 1.
\end{align*}
Hence,
$$
  \|H(t+h-k\tau)-H(t-k\tau)\|_{L^\infty(0,T-h)} \le h^{1/p}, \quad 1\le p\le\infty.
$$
Together with \eqref{aux4}, this finishes the proof.
\end{proof}

Theorem \ref{thm.dub} above and Lemma 4 in \cite{Sim87} imply the following proposition
involving the time derivative instead of the time shifts.

\begin{proposition}\label{prop.dub}
Let $B$, $Y$ be Banach spaces and $M_+$ be a seminormed nonnegative cone in $B$.
Let either $1\le p<\infty$, $r=1$ or $p=\infty$, $r>1$.
Assume that conditions {\rm (i)-(iii)} of Theorem \ref{thm.dub} hold and
$$
  \frac{\pa U}{\pa t}\mbox{ is bounded in }L^r(0,T;Y).
$$
Then $U$ is relatively compact in $L^p(0,T;B)$ (and in $C^0([0,T];B)$ if $p=\infty$).
\end{proposition}

\begin{proof}[Proof of Theorem \ref{thm.dub.pw}]
The case $1\le p<\infty$ is a consequence of Theorem \ref{thm.dub} and
Lemma \ref{lem.pw}. Therefore, let $p=\infty$.
We define the linear interpolations
$$
  \tilde u_\tau(t) = \left\{\begin{array}{ll}
	u_1 & \mbox{if }0<t\le \tau, \\
	u_k - (k\tau-t)(u_k-u_{k-1})/\tau & \mbox{if }(k-1)\tau<t\le k\tau,\
	2\le k\le N.
	\end{array}\right.
$$
Since $(k\tau-t)/\tau\le 1$ for $(k-1)\tau < t\le k\tau$, we have
$$
  \|\tilde u_\tau\|_{L^\infty(0,T;M_+)} \le 2\|u_\tau\|_{L^\infty(0,T;M_+)}
	\le C,
$$
using condition (iii). Furthermore, by condition (iv),
$$
  \left\|\frac{\pa \tilde u_\tau}{\pa t}\right\|_{L^r(0,T;Y)}
	= \frac{1}{\tau}\|\sigma_\tau u_\tau-u_\tau\|_{L^r(0,T-\tau;Y)} \le C.
$$
By Proposition \ref{prop.dub}, there exists a subsequence
$(\tilde u_{\tau'})$ of $(\tilde u_\tau)$ such that $\tilde u_{\tau'}\to \tilde u$
in $C^0([0,T];B)$ (and $\tilde u \in C^0([0,T];B)$). Applying Theorem
\ref{thm.dub.pw} with $p=1$ and $r=1$, there exists a subsequence
$(u_{\tau''})$ of $(u_{\tau'})$ such that $u_{\tau''}\to u$ in $L^1(0,T;B)$. Since
$$
  \|\tilde u_\tau-u_{\tau}\|_{L^1(0,T;B)}
	\le \|\sigma_\tau u_\tau-u_\tau\|_{L^1(0,T-\tau;B)} \le C\tau,
$$
it follows that $(\tilde u_{\tau''})$ and $(u_{\tau''})$ converge to the same
limit, implying that $\tilde u=u$. By the boundedness of $(u_\tau)$ in
$L^\infty(0,T;M_+)\subset L^\infty(0,T;B)$ and interpolation, we infer that
for $1\le q<\infty$, as $\tau\to 0$,
$$
  \|u_{\tau''}-u\|_{L^q(0,T;B)}
	\le \|u_{\tau''}-u\|_{L^1(0,T;B)}^{1/q}\|u_{\tau''}-u\|_{L^\infty(0,T;B)}^{1-1/q}
	\le C\|u_{\tau''}-u\|_{L^1(0,T;B)}^{1/q} \to 0.
$$
This shows that a subsequence of $(u_\tau)$ converges in $L^q(0,T;B)$
to a limit function $u\in C^0([0,T];$ $B)$.
\end{proof}


\subsection{Proof of Theorem \ref{thm.app}}\label{sec.app}

(a) We apply Theorem \ref{thm.dub.pw} to $B=L^{mr}(\Omega)$, $Y=(H^s(\Omega))'$,
and $M_+=\{u\ge 0:u^m\in W^{1,q}(\Omega)\}$ with $[u]=\|u^m\|_{W^{1,q}(\Omega)}^{1/m}$
for $u\in M_+$. Then $M_+$ is a seminormed nonnegative cone in $B$.
We claim that $M_+\hookrightarrow B$ compactly.
Indeed, it follows from the continuous embedding
$W^{1,q}(\Omega)\hookrightarrow L^r(\Omega)$ that for any $u\in M_+$,
$$\|u\|_{L^{mr}(\Omega)}=\|u^m\|^\frac{1}{m}_{L^{r}(\Omega)}\le C\|u^m\|^\frac{1}{m}_{W^{1,q}(\Omega)}=C[u].$$
Then $M_+\hookrightarrow B$ continuously. Let $(v_n)$ be bounded in $M_+$.
Then $(v_n^m)$ is bounded in $W^{1,q}(\Omega)$. Since $W^{1,q}(\Omega)$
embeddes compactly into $L^r(\Omega)$, up to a subsequence which is not relabeled,
$v_n^m\to z$ in $L^r(\Omega)$ with $z\ge 0$.
Again up to a subsequence, $v_n^m\to z$ a.e.\ and $v_n\to v:=z^{1/m}$ a.e. Hence $v_n^m\to v^m$ in $L^r(\Omega)$ which yields
$$
  \lim_{n\rightarrow\infty}\|v_n\|_{L^{mr}(\Omega)} = \lim_{n\rightarrow\infty}\|v_n^m\|_{L^r(\Omega)}^{1/m}
	=\|v^m\|_{L^r(\Omega)}^{1/m}=\|v\|_{L^{mr}(\Omega)}.
$$
Then it follows from Brezis-Lieb theorem (see \cite[p.~298, 4.7.30]{b07} or
\cite{bl83}) that $v_n\to v$ in $L^{mr}(\Omega)$
(for a subsequence). This proves the claim.
Next, let $w_n\to w$ in $L^{mr}(\Omega)$ and $w_n\to 0$ in $(H^\gamma(\Omega))'$.
Since $L^{mr}(\Omega)\hookrightarrow{\mathcal D}'(\Omega)$ and
$(H^\gamma(\Omega))'\hookrightarrow{\mathcal D}'(\Omega)$, the convergences
hold true in ${\mathcal D}'(\Omega)$ which gives $w=0$.
Furthermore, the following bound holds:
$$
  \|u_\tau\|_{L^{mp}(0,T;M_+)} = \|u_\tau^m\|_{L^p(0,T;W^{1,q}(\Omega))}^{1/m} \le C.
$$
By Theorem \ref{thm.dub.pw}, $(u_\tau)$ is relatively compact in
$L^{mp}(0,T;L^{mr}(\Omega))$.

(b) Note that the condition $\max\{0,(d-q)/(dq)\}<m<1+\min\{0,(d-q)/(dq)\}$
ensures that $s > 1$. By the first part of the proof, up to a subsequence,
$u_\tau\to u$ a.e. It is shown in the proof
of Proposition 2.1 in \cite{HiJu11} that this convergence and \eqref{ulogu}
imply that $u_\tau\to u$ in $L^\infty(0,T;L^1(\Omega))$. We infer from
the elementary inequality $|a-b|^{1/m}\le|a^{1/m}-b^{1/m}|$ for all $a$, $b\ge 0$
that
$$
  \|u_\tau^m-u^m\|_{L^\infty(0,T;L^{1/m}(\Omega))}
	\le \|u_\tau-u\|_{L^\infty(0,T;L^1(\Omega))}\to 0\quad\mbox{as }\tau\to 0.
$$
Then the Gagliardo-Nirenberg inequality gives
\begin{align*}
  \|u_\tau^m-u^m\|_{L^{p/m}(0,T;L^{s/m}(\Omega))}
	&\le C\|u_\tau^m-u^m\|_{L^p(0,T;W^{1,q}(\Omega))}^{m}
	\|u_\tau^m-u^m\|_{L^\infty(0,T;L^{1/m}(\Omega))}^{1-m} \\
	&\le C\|u_\tau^m-u^m\|_{L^\infty(0,T;L^{1/m}(\Omega))}^{1-m}\to 0.
\end{align*}
In particular, we infer that
$$
  \|u_\tau\|_{L^p(0,T;L^s(\Omega))}^m = \|u_\tau^m\|_{L^{p/m}(0,T;L^{s/m}(\Omega))}
	\le C.
$$
Furthermore, by the mean-value theorem,
$|a-b|\le \frac{1}{m}(a^{1-m}+b^{1-m})|a^m-b^m|$ for all $a$, $b\ge 0$, which
yields, together with the H\"older inequality,
\begin{align*}
  \|u_\tau-u\|_{L^p(0,T;L^s(\Omega))}
	&\le C\big(\|u_\tau\|_{L^p(0,T;L^s(\Omega))}^{1-m}
	+\|u\|_{L^p(0,T;L^s(\Omega))}^{1-m}\big)
	\|u_\tau^m-u^m\|_{L^{p/m}(0,T;L^{s/m}(\Omega))} \\
	&\le C\|u_\tau^m-u^m\|_{L^{p/m}(0,T;L^{s/m}(\Omega))}\to 0.
\end{align*}
This proves the theorem.


\section{Additional results}\label{sec.coro}

Using Lemma \ref{lem.pw}, we can specify Maitre's nonlinear compactness
result and Aubin-Lions lemma with intermediate spaces assumption for piecewise constant functions in time.

\begin{proposition}[Maitre nonlinear compactness]\label{prop.maitre}
Let either $1\le p<\infty$, $r=1$ or $p=\infty$, $r>1$.
Let $X$, $B$ be Banach spaces, and let $K:X\to B$ be a compact operator.
Furthermore, let $(v_\tau)\subset L^1(0,T;X)$ be a sequence of functions,
which are constant on each subinterval $((k-1)\tau,k\tau]$, $1\le k\le N$,
$T=N\tau$, and let $u_\tau=K(v_\tau)\in L^p(0,T;B)$. Assume that
\begin{enumerate}[{\rm (i)}]
\item $(v_\tau)$ is bounded in $L^1(0,T;X)$, $(u_\tau)$ is bounded in $L^1(0,T;B)$.
\item There exists $C>0$ such that for all $\tau>0$,
$\|\sigma_\tau u_\tau-u_\tau\|_{L^r(0,T-\tau;B)}\le C\tau$.
\end{enumerate}
Then, if $p<\infty$, $(u_\tau)$ is relatively compact in $L^p(0,T;B)$ and if
$p=\infty$, there exists a subsequence of $(u_\tau)$ converging in $L^q(0,T;B)$ for
all $1\le q<\infty$ to a limit function belonging to $C^0([0,T];B)$.
\end{proposition}

This result extends Theorem 1 in \cite{ChLi12}, which was proven for $r=p$ only,
for piecewise constant functions in time. In fact,
Lemma \ref{lem.pw} shows that condition (ii) implies a bound on
$\sigma_h u_\tau-u_\tau$ in $L^p(0,T-h;B)$, and Theorem 1 in \cite{ChLi12} applies
for $p<\infty$. The case $p=\infty$ is treated as in the proof of
Theorem \ref{thm.dub.pw}.

\begin{proposition}[Aubin-Lions compactness]\label{th4}
Let $X,\,B,\,Y$ be Banach spaces and $1\leq p<\infty$. Assume that
$X\hookrightarrow Y$ compactly, $X\hookrightarrow B\hookrightarrow Y$ continuously and there exist $\theta\in (0,1)$, $C_\theta>0$ such that
for any $u\in X,$ $\|u\|_{B}\le C_\theta\|u\|_{X}^{1-\theta}\|u\|_{Y}^\theta.$ Furthermore, let $(u_\tau)$ {be a sequence of functions,
which are constant on each subinterval $((k-1)\tau,k\tau]$, $1\le k\le N$, {$T=N\tau$}}.
If
\begin{enumerate}[{\rm (i)}]
\item $(u_\tau)$ is bounded in $L^{p}(0,T;X)$.
\item There exists $C>0$ such that for all $\tau>0$,
$\|\sigma_\tau u_\tau-u_\tau\|_{L^1(0,T-\tau;B)}\le C\tau$.
\end{enumerate}
Then $\{u_\tau\}$ is relatively compact in
$L^q(0,T;B)$ for all $p\le q<p/(1-\theta)$.
\end{proposition}

\begin{proof}
Let $p\le q<p/(1-\theta)$ and set $\ell=\theta/(1/q-(1-\theta)/p)$. Then
$\ell\in [1,\infty)$ and $1/q=(1-\theta)/p+\theta/\ell$.
Hence it follows from Lemma \ref{lem.pw} that
$\|\sigma_h u_\tau-u_\tau\|_{L^\ell(0,T-h;Y)}\leq C h^{1/\ell}$ for all $0<h<T.$
This and Theorem 7 of \cite{Sim87} prove the result.
\end{proof}

This result improves Theorem 1 in \cite{DrJu12} for the case $p<\infty$.
For piecewise constant functions, Lemma \ref{lem.pw} can be applied
to Theorem 1.1 of \cite{Ama00} which yields another compactness result.

In finite-element or finite-volume approximation,
$u_n\in Y_n$ may be the solution of a discretized evolution equation,
where $(Y_n)$ is a sequence of (finite-dimensional) Banach spaces which
``approximates'' the (infinite-dimensional) Banach space $Y$. Since the
spaces $Y_n$ depend on the index $n$, the classical Aubin-Lions lemma generally
does not apply. Gallou\"et and Latch\'e \cite{GaLa12} have proved a discrete version
of this lemma. We generalize their result for seminormed cones $M_n$ and
allow for the case $p=\infty$.

\begin{proposition}[Discrete Aubin-Lions-Dubinski\u{\i}]
\label{thm.dub.disc}
Let $B$, $Y_n$ be Banach spaces $(n\in\N)$ and let $M_n$ be seminormed
nonnegative cones in $B$ with ``semiorms'' $[\cdot]_n$. Let $1\le p\le\infty$.
Assume that
\begin{enumerate}[{\rm (i)}]
\item $(u_n)\subset L^p(0,T;M_n\cap Y_n)$ and there exists $C>0$ such that
$\|u_n\|_{L^p(0,T;M_n)}\le C$.
\item $\|\sigma_h u_n-u_n\|_{L^p(0,T-h;Y_n)}\to 0$ as $h\to 0$, uniformly in $n\in\N$.
\end{enumerate}
Then $(u_n)$ is relatively compact in $L^p(0,T;B)$ (and in $C^0([0,T];B)$ if
$p=\infty)$.
\end{proposition}

\begin{proof}
The proof uses the same techniques as in Section \ref{sec.pro}, therefore we
give only a sketch. Similarly as in Lemma \ref{lem.dub}, a Ehrling-type
inequality holds: Let $u_n\in M_n$ ($n\in\N$). Assume that (i) if $[u_n]_n\le C$
for all $n\in\N$, for some $C>0$, then $(u_n)$ is relatively compact in $B$;
(ii) if $u_n\to u$ in $B$ as $n\to\infty$ and $\lim_{n\to\infty}\|u_n\|_{Y_n}=0$
then $u=0$. Then for all $\eps>0$, there exists $C_\eps>0$ such that
for all $n\in\N$, $u$, $v\in M_n\cap Y_n$,
$$
  \|u-v\|_B \le \eps([u]_n+[v]_n) + C_\eps\|u-v\|_{Y_n}.
$$
We infer as in the proof of Theorem \ref{thm.dub} that conditions (i) and
(ii) imply that
$$
  \|\sigma_h u_n-u_n\|_{L^p(0,T-h;B)}\to 0\quad\mbox{as }h\to 0,\quad
	\mbox{uniformly for }n\in\N.
$$
Finally, as in the proof of Lemma 6 in \cite{ChLi12}, the relative compactness
of $(u_n)$ in $L^p(0,T;B)$ (and in $C^0([0,T];B)$ if $p=\infty)$ follows.
\end{proof}


\end{document}